\documentclass[12pt]{amsart}     
\usepackage{amsmath}

\usepackage{amsfonts}
\usepackage{amssymb,enumerate}

\usepackage{tikz}
\usetikzlibrary{arrows}
\usepackage{mathrsfs}
\usetikzlibrary{arrows}
\usepackage{listings}
\definecolor{DimGray}{rgb}{0.41, 0.41, 0.41}

\newtheorem{theorem}{Theorem}[section]
\theoremstyle{definition}
\newtheorem{definition}[theorem]{Definition}

\newtheorem{example}[theorem]{Example}
\theoremstyle{plain}
\newtheorem{proposition}[theorem]{Proposition}
\newtheorem*{theorem*}{Theorem}
\newtheorem{lemma}[theorem]{Lemma}
\newtheorem{corollary}[theorem]{Corollary}

\theoremstyle{remark}

\numberwithin{equation}{section}
\numberwithin{figure}{section}

\title[A formula for the conductor of a semimodule of some numerical semigroups]{A formula for the conductor of a semimodule of a numerical semigroup with two generators}
\author{Patricio Almir\'on}
\author{Julio-Jos\'e Moyano-Fern\'andez}
\dedicatory{Dedicated to the memory of Fernando Eduardo Torres Orihuela}



\thanks{The first author was partially supported by Spanish Goverment, Ministerios de Ciencia e Innovaci\'on y de Universidades MTM2016-76868-C2-1-P. The second author was partially supported by the Spanish Government, Ministerios de Ciencia e Innovaci\'on y de Universidades, grant PGC2018-096446-B-C22, as well as by Universitat Jaume I, grant UJI-B2018-10.}

\address{Instituto de Matemática interdisciplinar (IMI) y departamento de \'{A}lgebra, Geometr\'{i}a y Topolog\'{i}a\\
	Facultad de Ciencias Matem\'{a}ticas\\
	Universidad Complutense de Madrid\\
	28040, Madrid, Spain.}
\email{palmiron@ucm.es}

\address{Universitat Jaume I, Campus de Riu Sec, Departamento de Matem\'aticas \& Institut Universitari de Matem\`atiques i Aplicacions de Castell\'o, 12071
	Caste\-ll\'on de la Plana, Spain}

\email{moyano@uji.es}

\begin{document}

\begin{abstract}
We provide an expression for the conductor $c(\Delta)$ of a semimodule $\Delta$ of a numerical semigroup $\Gamma$ with two generators in terms of the syzygy module of $\Delta$ and the generators of the semigroup. In particular, we deduce that the difference between the conductor of the semimodule and the conductor of the semigroup is an element of $\Gamma$, as well as a formula for $c(\Delta)$ in terms of the dual semimodule of $\Delta$.
\end{abstract}
\maketitle
\section{Introduction}
\label{intro}
A classical problem in the combinatorics of natural numbers is to find a closed expression for the largest natural number that is not representable as a nonnegative linear combination of some relatively prime numbers, called the Frobenius number. This problem can be encoded in terms of getting a formula for the conductor of a numerical semigroup; it is known under the name ``Frobenius problem''.

Consider $\mathbb{N}=\{x\in \mathbb{Z}: x\geq 0\}$. A numerical semigroup $\Gamma$ is an additive sub-monoid of the monoid $(\mathbb{N},+)$ such that the greatest common divisor of all its elements is equal to $1$. The complement $\mathbb{N}\setminus \Gamma$ is therefore finite, and its elements are called gaps of $\Gamma$. Moreover, $\Gamma$ is finitely generated and it is not difficult to find a minimal system of generators of $\Gamma$, see. e.g. Rosales and Garc\'ia S\'anchez \cite{RosalesGarciaSanchez}.

The number $c(\Gamma)=\max (\mathbb{N}\setminus \Gamma)+1$ is called the conductor of $\Gamma$; in particular $c(\Gamma)-1$ is the Frobenius number of $\Gamma$. The computation of $c(\Gamma)$ for an arbitrary number of minimal generators of $\Gamma$ is NP-hard (see Ram\'irez Alfons\'in \cite{ram} for a good account of this), but there are some special cases in which a closed formula is available. For example, if $\Gamma=\alpha \mathbb{N}+\beta\mathbb{N}:=\langle \alpha , \beta \rangle$, then $c(\Gamma)=\alpha \beta -\alpha -\beta +1$. However, for a numerical semigroup with more than two generators it is not possible in general to obtain a closed polynomial formula for its conductor in terms of the minimal set of generators (see Curtis \cite{curtis}).

We are interested in subsets of $\mathbb{N}$ which have an additive structure over $\Gamma$ (in analogy with the structure of module over a ring): a $\Gamma$-semimodule is a non-empty subset $\Delta$ of $\mathbb{N}$ with $\Delta+\Gamma\subseteq \Delta$. A system of generators of $\Delta$ is a subset $\mathcal{E}$ of $\Delta$ such that $\Delta=\bigcup_{x\in \mathcal{E}} (x+\Gamma)$; it is called minimal if no proper subset of $\mathcal{E}$ generates $\Delta$. Notice that, since $\Delta\setminus \Gamma$ is finite, every $\Gamma$-semimodule is finitely generated and has a conductor
$$
c(\Delta)=\max (\mathbb{N}\setminus \Delta)+1.
$$

Motivated by the Frobenius problem, it is natural to ask for a closed formula for the conductor of a \(\Gamma\)-semimodule. The purpose of this note is to give a formula for $c(\Delta)$ in the case $\Gamma=\langle \alpha, \beta \rangle$ in terms of the generators of the semimodule of syzygies of $\Delta$, see \cite{MU1}, as well as in terms of the generators of the dual of this semimodule, see \cite{MU2}. These are the contents of our two main results, namely Theorem \ref{formula} resp.~Corollary \ref{cor:main}.

\section{Semimodules over a numerical semigroup}

Let $\Gamma$ be a numerical semigroup. This section is devoted to collect the main properties concerning $\Gamma$-semimodules. The reader is referred to \cite{RosalesGarciaSanchez} or \cite{ram} for specific material about numerical semigroups.

Every $\Gamma$-semimodule $\Delta$ has a unique minimal system of generators (see e.g. \cite[Lemma 2.1]{MU1}). Two $\Gamma$-semimodules $\Delta$ and $\Delta'$ are called isomorphic if there is an integer $n$ such that $x\mapsto x+n$ is a bijection from $\Delta$ to $\Delta'$; we write then $\Delta\cong \Delta'$. For every $\Gamma$-semimodule $\Delta$ there is a unique semimodule $\Delta' \cong \Delta$ containing $0$; such a semimodule is called normalized. 
Moreover, the minimal system of generators $\{x_0=0,\ldots , x_n\}$ of a normalized $\Gamma$-semimodule is a $\Gamma$-lean set, i.e. it satisfies that
$$
|x_i-x_j| \notin \Gamma \ \ \mbox{for~any} \ \ 0\leq i <j \leq n,
$$
and conversely, every $\Gamma$-lean set of $\mathbb{N}$ minimally generates a normalized $\Gamma$-semimo\-dule. Hence there is a bijection between the set of isomorphism classes of $\Gamma$-semimodules and the set of $\Gamma$-lean sets of $\mathbb{N}$. See  Sect. 2 in \cite{MU1} for the proofs of those statements.

There is another kind of system of generators---not minimal---for a semimodule $\Delta$ of $\Gamma$ relative to $s\in \Gamma\setminus \{0\}$: this is the set of the $s$ smallest elements in $\Delta$ in each of the $s$ classes modulo $s$, namely the set $\Delta \setminus (s+\Delta)$, and is called the Ap\'ery set of $\Delta$ with respect to $s$; we write $\mathrm{Ap}(\Delta,s)$. 

A formula for the conductor in terms of $\mathrm{Ap}(\Delta,s)$ for $s\in \Gamma\setminus \{0\}$ is easily deduced.

\begin{proposition}\label{prop:cond_Apery}
	Let $\Delta$ be a $\Gamma$-semimodule. For any $s\in \Gamma\setminus\{0\}$ we have that
	$$
	c(\Delta)-1=\mathrm{max}_{\leq_\mathbb{N}} \mathrm{Ap}(\Delta,s)-s.
	$$
\end{proposition}

\begin{proof}
	The equality follows as in the case $\Delta=\Gamma$, see e.g.~Lemma 3 in Brauer and Shockley \cite{bs}.
\end{proof}


In this paper we will consider numerical semigroups with two generators, say $\Gamma=\langle \alpha, \beta \rangle$, with $\alpha,\beta \in \mathbb{N}$ with $\alpha < \beta $ and $\mathrm{gcd}(\alpha, \beta) = 1$.
As mentioned above, the conductor of $\Gamma$ can be expressed as $c=c(\langle \alpha,\beta \rangle)=(\alpha-1)(\beta-1)$. The gaps of $\langle \alpha, \beta \rangle$ are also easy to describe: they admit a unique representation $\alpha \beta -a\alpha -b \beta$, where $a\in \ ]0,\beta-1]\cap \mathbb{N}$ and $b\in \ ]0,\alpha-1]\cap \mathbb{N}$. This writing yields a map from the set of gaps of $\langle \alpha, \beta \rangle$ to $\mathbb{N}^2$ given by 
$$\alpha \beta -a\alpha -b \beta \mapsto (a,b),$$ which allows us to identify a gap with a lattice point in the lattice $\mathcal{L}=\mathbb{N}^2$; since the gaps are positive numbers, the point lies inside the triangle with vertices $(0,0),(0,\alpha),(\beta, 0)$. 

In the following we will use the notation
$$
e=\alpha\beta - a(e)\alpha-b(e)\beta
$$
for a gap $e$ of the semigroup $\langle \alpha,\beta \rangle$; if the gap is subscripted as $e_i$ then we write $a_i=a(e_i)$ and $b_i=b(e_i)$. 

Let us denote by \(\leq\) the total ordering in $\mathbb{N}$; sometimes we will write \(\leq_{\mathbb{N}}\) to emphasize that it is the natural ordering. In addition, we define the following partial ordering $\preceq$ on the set of gaps:

\begin{definition}\label{jorder}
	Given two gaps $e_1,e_2$ of $\langle \alpha , \beta \rangle$, we define 
	$$
	e_1 \preceq e_2 \ \ :\Longleftrightarrow \ a_1\leq a_2 \ \ \wedge \ \ b_1 \geq b_2
	$$
	and
	$$
	e_1 \prec e_2 \ \  :\Longleftrightarrow  \ a_1 < a_2 \ \ \wedge \ \ b_1 >b_2.
	$$
\end{definition}

Observe that the ordering $\preceq$ differs from the one used by the second author and Uliczka in \cite{MU,MU1,MU2}: there the gaps $e_i$ are ordered by decreasing sequence of the corresponding $a_i$.

Let $\mathcal{E}=\{0,e_1,\ldots , e_n\} \subseteq \mathbb{N}$ with gaps $e_i=\alpha\beta -a_i \alpha -b_i\beta$ of $\langle \alpha, \beta  \rangle$ for every $i=1,\ldots, n$ such that $a_1<a_2<\cdots < a_n$. Corollary 3.3 in \cite{MU1} ensures that $\mathcal{E}$ is $\langle \alpha , \beta \rangle$-lean if and only if $b_1>b_2>\cdots > b_n$. 

This simple fact leads to an identification (cf. \cite[Lemma 3.4]{MU1}) between an $\langle \alpha, \beta \rangle$-lean set and a lattice path with steps downwards and to the right from $(0,\alpha)$ to $(\beta,0)$ not crossing the line joining these two points, where the lattice points identified with the gaps in $\mathcal{E}$ mark the turns from the $x$-direction to the $y$-direction; these turns will be called ES-turns for abbreviation. Figure \ref{fig1} shows the lattice path corresponding to the $\langle 5,7 \rangle$-lean set $\{0,9,11,8\}$.

\begin{center}
	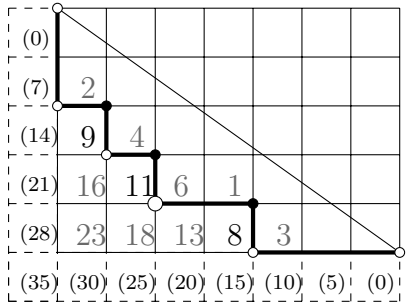
\begin{figure}
		\begin{tikzpicture}[scale=0.65]
		\draw[dashed] (-1,-1) grid [step=1cm](7,0);
		\draw[dashed] (-1,-1) grid [step=1cm](0,5);
		\draw[] (0,0) grid [step=1cm](7,5);
		\draw[] (0,5) -- (7,0);
		\draw[ultra thick] (0,5) -- (0,3) -- (1,3) -- (1,2) -- (2,2) -- (2,1) -- (4,1) -- (4,0) -- (7,0);
		
		\draw[fill] (1,3) circle [radius=0.1]; 
		\draw[fill] (2,2) circle [radius=0.1]; 
		\draw[fill] (4,1) circle [radius=0.1];

		\node [below right][DimGray] at (0.15,0.8) {$23$};
		\node [below right][DimGray] at (1.15,0.8) {$18$};
		\node [below right][DimGray] at (2.15,0.8) {$13$};
		\node [below right] at (3.25,0.8) {$8$};
		\node [below right][DimGray] at (4.25,0.8) {$3$};
		
		\node [below right][DimGray] at (0.15,1.8) {$16$};
		\node [below right] at (1.15,1.8) {$11$};
		\node [below right][DimGray] at (2.15,1.8) {$6$};
		\node [below right][DimGray] at (3.25,1.8) {$1$};
		
		\node [below right] at (0.25,2.8) {$9$};
		\node [below right][DimGray] at (1.25,2.8) {$4$};
		
		\node [below right][DimGray] at (0.25,3.8) {$2$};

		\node [below right] at (-1,-0.2) {$\scriptstyle (35)$};
		\node [below right] at (0,-0.2) {$\scriptstyle (30)$};
		\node [below right] at (1,-0.2) {$\scriptstyle (25)$};
		\node [below right] at (2,-0.2) {$\scriptstyle (20)$};
		\node [below right] at (3,-0.2) {$\scriptstyle (15)$};
		\node [below right] at (4,-0.2) {$\scriptstyle (10)$};
		\node [below right] at (5.1,-0.2) {$\scriptstyle (5)$};
		\node [below right] at (6.1,-0.2) {$\scriptstyle (0)$};
		
		\node [below right] at (-1,0.8) {$\scriptstyle (28)$};
		\node [below right] at (-1,1.8) {$\scriptstyle (21)$};
		\node [below right] at (-1,2.8) {$\scriptstyle (14)$};
		\node [below right] at (-0.95,3.8) {$\scriptstyle (7)$};
		\node [below right] at (-0.95,4.8) {$\scriptstyle (0)$};
		\draw[fill=white] (0,3) circle [radius=0.1]; 
		\draw[fill=white] (1,2) circle [radius=0.1]; 
		\draw[fill=white] (2,1) circle [radius=0.15]; 
		\draw[fill=white] (4,0) circle [radius=0.1]; 
		\draw[fill=white] (0,5) circle [radius=0.1]; 
		\draw[fill=white] (7,0) circle [radius=0.1]; 
		\end{tikzpicture}
		\caption{Lattice path for the $\langle 5,7 \rangle$-lean set $I=[0,9,11,8]$ and the corresponding syzygy minimal generators $J=[14,16,18,15]$. The biggest generator $M$ with respect to $\leq_{\mathbb{N}}$ is depicted bigger.}\label{fig1}
	\end{figure}
\end{center}

Let $g_0=0,g_1,\ldots , g_n$ be the minimal system of generators of a $\langle \alpha, \beta \rangle$-semi\-module $\Delta$. From now on, we will assume that the indexing in the minimal set of generators of $\Delta$ is such that $g_0=0\preceq g_1\preceq\cdots\preceq g_n$; accordingly we will use the notation $[g_0,\ldots, g_n]$ rather than $\{g_0,\ldots , g_n\}$. In \cite{MU1} it was introduced the notion of syzygy of $\Delta$ as the $\langle \alpha, \beta \rangle$-semimodule
$$
\mathrm{Syz}(\Delta):=\bigcup_{i,j\in \{0,\ldots n\}, i\neq j} \Big ((\Gamma + g_i)\cap (\Gamma + g_j) \Big ).
$$

The semimodule of syzygies of the semimodule $\Delta$ minimally generated by $[g_0=0,g_1,\dots,g_n]$ can be characterized as follows (see \cite[Theorem 4.2]{MU1}; since Definition \ref{jorder} differs from the corresponding in \cite{MU,MU1,MU2}---as mentioned above, Definition \ref{defin:syz} must be conveniently adapted here):
\begin{definition}\label{defin:syz}
	\begin{equation*}
	\mathrm{Syz}(\Delta) =\bigcup_{0\leq k<j\leq n} \Big ((\Gamma + g_k)\cap (\Gamma + g_j) \Big )= \bigcup_{k=0}^{n} (\Gamma + h_k),
	\end{equation*}
	
	where $h_1,\ldots , h_{n-1}$ are gaps of $\Gamma$, $h_0,h_n \leq \alpha \beta$, and
	\begin{align*}
	&h_k \equiv g_k ~\mathrm{ mod }~ \beta, ~h_ k > g_{k}  \ \mbox{for } k=0,\ldots , n \\
	&h_k \equiv g_{k+1} ~ \mathrm{mod } ~\alpha,~ h_ k > g_{k+1} \ \mbox{for } k=0,\ldots , n-1\\
	&h_n \equiv 0~\mathrm{mod } ~\alpha, \ \mathrm{and } ~h_n \geq 0
	\end{align*}
\end{definition} 

In particular, \(J=[h_0,\dots,h_n]\) is a minimal system of generators of the semimodule \(\mathrm{Syz}(\Delta)\), hence \(h_0\preceq h_1\preceq\cdots\preceq h_n.\) Therefore it is easily seen that the SE-turns of the lattice path associated to \(\Delta\) can be identified with the minimal set of generators of the syzygy module (we call SE-turns to the turns from the \(y\)--direction to the \(x\)--direction). After that, we can associate to any \(\Gamma\)-semimodule \(\Delta\) a lean set \([I,J]\), where \(I\) is a minimal set of generators of \(\Delta\) and \(J\) a minimal set of generators of \(\mathrm{Syz}(\Delta);\) or, equivalently, a lattice path. An easy consequence of this fact is the following lemma.

\begin{lemma}\label{lem:aux}
	Let \(\Delta\) be a \(\Gamma\)-semimodule with associated $\Gamma$-lean set \([I,J]\) for \(I=[g_0=0,g_1,\dots,g_n]\) and \(J=[h_0,\dots,h_n]\). Then, for any \(h\in J\) we have \(h-\alpha-\beta\notin\Delta.\)
\end{lemma}

\begin{proof}
	Consider \(h\in J\) such that that \(g_i\prec h\prec g_{i+1}\). Let us denote $(a_j,b_j)$ resp. $(a_{j+1},b_{j+1})$ the coordinates of $g_j$ resp. $g_{j+1}$ in the lattice $\mathcal{L}$; then the element \(h\) is represented in the lattice path as $(a_j,b_{j+1})$, see Definition \ref{defin:syz}. By contradiction, assume that $h-\alpha-\beta\in \Delta$; then there exists a gap \(g\in I\) together with two integers $\nu_1,\nu_2\in\mathbb{N}$ such that 
	\[
	h-\alpha-\beta=\nu_1\alpha+\nu_2\beta+g.
	\]
	Since $h-\alpha-\beta\notin\Gamma$, we may write 
	\[
	h-\alpha-\beta=\alpha\beta-(a_j+1)\alpha-(b_{j+1}+1)\beta.
	\]
	The writing of $g$ as $g=\alpha\beta-a\alpha-b\beta$ is unique whenever $(a,b)\in \mathcal{L}$, therefore
	\[
	a_j+1=a-\nu_1,\quad b_{j+1}+1=b-\nu_2.
	\]
	These equalities yield the conditions $a_j<a$ and \(b_{j+1}<b\). But the unique minimal generator which fulfills these conditions is \(g_{j+1}\); however, \(h\) cannot be expressed as \(h=g_{j+1}+\nu+\alpha+\beta\) since \(h\) is represented in the lattice path as $(a_j,b_{j+1})$, a contradiction.
\end{proof}
\begin{example}
	For \(\Gamma=\langle 5,7 \rangle\) and the \(\Gamma\)--semimodule \(\Delta_I\) minimally generated by $I=[0,9,11,8]$, it is easily deduced that the syzygy module $\mathrm{Syz}(\Delta_I)$ is minimally generated by $J=[14,16,18,15]$, cf.~Figure \ref{fig1}; there we have extended the labelling beyond the axis in the natural way in order to have also an interpretation of $J$ in terms of the lattice path.
	Observe that by Lemma \ref{lem:aux} we have \(14-7-5=2\notin \Delta,\) \(16-7-5=4\notin \Delta,\) \(18-7-5=6\notin \Delta\) and \(15-7-5=3\notin \Delta\); this can be read off from Figure \ref{fig1} as well. 
\end{example}

\section{A formula for the conductor of an $\langle \alpha, \beta \rangle$-semimodule}

In this section we are going to provide a formula for the conductor of a \(\Gamma\)-semimodule with any number of generators in terms of the generators of $\Gamma$ and a special syzygy of the $\Gamma$-semimodule. In particular, we will obtain some relations between the conductor of $\Gamma$ and the conductor of the $\Gamma$-semimodule. Finally, we will provide a formula for the conductor of the \(\Gamma\)-semimodule in terms of its dual.

\begin{theorem}\label{formula}
	Let \(\Delta\) be a \(\Gamma\)-semimodule with associated lean set \([I,J]\) as above, and let \(M:=\max_{\leq_\mathbb{N}}\{h\in J\}\) denote the biggest (with respect to the total ordering of the natural numbers) minimal generator of $\mathrm{Syz}(\Delta)$. Then
	\[c(\Delta)=M-\alpha-\beta+1.\]
	In particular, if $(m_1,m_2)$ are the coordinates of the point representing $M$ in the lattice $\mathcal{L}$, then we have
	\[c(\Delta)=c(\Gamma)-m_1\alpha-m_2\beta.\]
\end{theorem}

\begin{proof}
	Since \(c(\Delta)-1\) is the Frobenius number of the \(\Gamma\)--semimodule \(\Delta\), it is enough to check that (i) \(M-\alpha-\beta\notin\Delta\), and (ii) if \(\ell\notin\Delta\), then \(\ell\leq M-\alpha-\beta.\) 
	The statement (i) is clear by Lemma \ref{lem:aux}, since \(M\in J\). To see (ii), consider an element \(\ell\notin\Delta\), which in particular means \(\ell\notin\Gamma\). So we can associate to $\ell$ a point $(a,b)$ in the lattice $\mathcal{L}$. Moreover, \(\ell\) is upon and not contained in the lattice path associated to $I$. This means that there exists some \(j\in J\) with coordinates \((j_1,j_2)\) in the lattice path such that \(a> j_1\) and \(b> j_2\), otherwise \(\ell\) would be an element of \(\Delta\), since the elements represented by lattice points on and under the lattice path belong to \(\Delta\). Therefore, \(a\geq j_1+1\) and \(b\geq j_2+1\). Thus, from the representation of \(\ell\) and \(j\) as gaps we can check that
	
	\[\ell=\alpha\beta-a\alpha-b\beta\leq_{\mathbb{N}}\alpha\beta-(j_1+1)\alpha-(j_2+1)\beta=j-\alpha-\beta.\] 
	
	Hence, since \(M=\max_{\leq_\mathbb{N}}\{h\in J\}\) and  \(M\in J\), we have that \(M-\alpha-\beta\geq_{\mathbb{N}}\ell\) for any \(\ell\notin\Delta\), which proves (ii). 
	
	Finally, since $M$ can be represented as a lattice point $(m_1,m_2)\in \mathcal{L}$, we have
	\[c(\Delta)=M-\alpha-\beta+1=\alpha\beta-m_1\alpha-m_2\beta-\alpha-\beta+1=c(\Gamma)-m_1\alpha-m_2\beta.\]
\end{proof}

\begin{example}\label{example}
	Again in the case of \(\Gamma=\langle 5,7 \rangle\) and the \(\Gamma\)--semimodule minimally generated by $[0,9,11,8]$, Figure \ref{fig1} illustrates that the maximal syzygy is \(M=18\), and so the conductor of the semimodule is $c(\Gamma)-5 m_1-7 m_2 = 24-5\cdot 2 - 7 \cdot 1=7$. 
\end{example}

Notice that for the particular case of $\Delta=\Gamma$ we have $M=\alpha\beta$, and we recover the well-known formula $c(\Gamma)=\alpha\beta-\alpha-\beta +1$.
The value $M$ can be easily characterized in terms of the Ap\'ery set of $\Delta$ with respect to $\alpha+\beta$:

\begin{proposition}
	Let \(M:=\max_{\leq_\mathbb{N}}\{h\in J\}\) be the biggest minimal generator of the syzygy module with respect to the natural ordering of $\mathbb{N}$ as above, then
	$$
	M=\max_{\leq_\mathbb{N}} \mathrm{Ap}(\Delta,\alpha+\beta).
	$$
\end{proposition}

\begin{proof}
	This is a consequence of Proposition \ref{prop:cond_Apery} for $s=\alpha+\beta \in \langle \alpha, \beta \rangle$.
\end{proof}

A straightforward consequence of Theorem \ref{formula} is the following.

\begin{corollary}
	Let \(\Delta\) be a \(\Gamma\) semimodule. Then 
	\[c(\Gamma)-c(\Delta)\in\Gamma.\]
\end{corollary}

We conclude this paper rewriting the formula of Theorem \ref{formula} in terms of the dual \(\Gamma\)--semimodule of \(\Delta\),
\[
\Delta^{\ast}:=\{z\in\mathbb{Z}\;|\;z+\Delta\subset\Gamma\},
\]
see \cite{MU2}. An important fact about the dual semimodule is that the minimal set of generators of \(\mathrm{Syz}(\Delta)\)  is in bijection with the minimal set of generators of \(\Delta^{\ast}\):
\begin{lemma}[\cite{MU2}, Lemma 6.1]\label{lemma:61}
	The minimal sets of generators of \(\Delta^{\ast}\) and \(\mathrm{Syz}(\Delta)\) are in correspondence via the map \(x\mapsto\alpha\beta-x.\)
\end{lemma}
In particular, this bijection together with Theorem \ref{formula} allows us to compute the conductor of the semimodule \(\Delta\) in terms of the minimal generators of $\Delta^{\ast}$ in a natural way:
\begin{corollary}\label{cor:main}
	Let $\Delta$ be a $\Gamma$-semimodule, and let $\Delta^{\ast}$ be its dual, minimally generated by $x_0,\dots,x_n$. Then
	\[c(\Delta)=\alpha\beta-\min_{\leq\mathbb{N}}\{x_0,\ldots, x_n\}-\alpha-\beta+1.\]
\end{corollary}
\begin{proof}
	By Theorem \ref{formula} we have that \(c(\Delta)=\max_{\leq_\mathbb{N}}\{h\in J\}-\alpha-\beta+1, \) where \(J\) is a minimal set of generators of \(\mathrm{Syz}(\Delta).\) Lemma \ref{lemma:61} yields the equality 
	$$
	\min_{\leq\mathbb{N}}\{x_0, x_1,\ldots, x_n\}=\alpha\beta-\max_{\leq_\mathbb{N}}\{h\in J\},
	$$ which allows us to conclude.
\end{proof}

\begin{example}
	By \cite[Theorem 2.5]{MU2}, the minimal generators of the dual of the $\langle 5,7 \rangle$-semimodule $\Delta_I$ are given by $[20,17,19,21]$; notice that, for the explicit calculation, the mentioned theorem requires the reverse ordering $\succeq$ instead of the ordering $\preceq$ we use here. The minimum of this set is $17$, therefore by Corollary \ref{cor:main} we have $c(\Delta)=35-17-12+1=7$, as computed in Example \ref{example}.
\end{example}

\end{document}